\documentclass[12pt]{amsart}
\usepackage{latexsym, amsmath, amscd, amssymb, amsthm, bm, comment}
\usepackage[colorlinks=true,linkcolor=blue,urlcolor=blue, citecolor=blue]%
  {hyperref}
\usepackage[shortalphabetic]{amsrefs}
\usepackage[T1]{fontenc}
\usepackage{mathptmx}
\usepackage{microtype}

\usepackage[centering, includeheadfoot, hmargin=1.0in, vmargin=0.8in,
  headheight=30.4pt]{geometry}
\newtheorem{lemma}{Lemma}[section]
\newtheorem{theorem}[lemma]{Theorem}

\newtheorem{proposition}[lemma]{Proposition}

\theoremstyle{definition}
\newtheorem{remark}[lemma]{Remark}
\newtheorem{example}[lemma]{Example}
\theoremstyle{remark}
\newcommand{\define}[1]{{\bfseries\itshape #1}}

\makeatletter
\def\subsection{\@startsection{subsection}{2}%
  \z@{.5\linespacing\@plus.7\linespacing}{-.5em}%
  {\normalfont\scshape}}
\def\subsubsection{\@startsection{subsection}{2}%
  \z@{.5\linespacing\@plus.7\linespacing}{-.5em}%
  {\normalfont\bfseries}}
\makeatother
  
\newcommand{\relphantom}[1]{\mathrel{\phantom{#1}}}
 
\renewcommand{\AA}{\ensuremath{\mathbb{A}}} 
\newcommand{\CC}{\ensuremath{\mathbb{C}}} 
\newcommand{\NN}{\ensuremath{\mathbb{N}}} 
\newcommand{\PP}{\ensuremath{\mathbb{P}}} 
 
\newcommand{\RR}{\ensuremath{\mathbb{R}}} 
\newcommand{\ZZ}{\ensuremath{\mathbb{Z}}} 
\newcommand{\eulerian}[2]{\genfrac{\langle}{\rangle}{0pt}{}{#1}{#2}}
\renewcommand{\geq}{\geqslant}
\renewcommand{\leq}{\leqslant}
\newcommand{\iso}{\cong}

\DeclareMathOperator{\conv}{conv}
\DeclareMathOperator{\interior}{int}
\DeclareMathOperator{\vol}{vol}

\begin{document}

\title{Log-concavity of asymptotic multigraded Hilbert series}

\author[A. McCabe]{Adam McCabe}
\address{Adam McCabe\\ 
  Department of Mathematics\\ 
  University of Toronto \\
  40 St.~George Street\\
  Toronto\\ 
  ON\\
  M5S~2E4\\ 
  Canada}
\email{\href{mailto:adam.r.mccabe@gmail.com}{adam.r.mccabe@gmail.com}}

\author[G.G.~Smith]{Gregory G. Smith}
\address{Gregory G. Smith\\ 
  Department of Mathematics \& Statistics\\ 
  Queen's University\\ 
  Kingston\\ 
  ON \\ 
  K7L~3N6\\ 
  Canada}
\email{\href{mailto:ggsmith@mast.queensu.ca}{ggsmith@mast.queensu.ca}}



\begin{abstract}
  We study the linear map sending the numerator of the rational function
  representing the Hilbert series of a module to that of its $r$-th Veronese
  submodule.  We show that the asymptotic behaviour as $r$ tends to infinity
  depends on the multidegree of the module and the underlying positively
  multigraded polynomial ring.  More importantly, we give a polyhedral
  description for the asymptotic polynomial and prove that the coefficients are
  log-concave.
\end{abstract}

\maketitle

\vspace*{-1em}

\section{Introduction} 
\label{sec:intro}

\subpdfbookmark{Statement}{sub:statement} 

\noindent
Although motivated by multigraded Hilbert series, our main result only involves
linear operators on a multivariate power series.  To be explicit, let $A := [
\bm{a}_1 \, \dotsb \,\, \bm{a}_n ]$ be an integer $(d \times n)$-matrix of rank
$d$ such that the only nonnegative vector in the kernel is the zero vector
({\footnotesize i.e.{} $\bm{a}_1, \dotsc, \bm{a}_n$ is an acyclic vector
  configuration; see \cite{Ziegler}*{\S6.2}}).  Equivalently, the rational
function $1 / \bigl( \prod\nolimits_{1 \leq j \leq n} (1-\bm{t}^{\bm{a}_j})
\bigr)$ has a unique expansion as a multivariate formal power series; cf.{}
Lemma~8.16 in \cite{MillerSturmfels}.  For each positive integer $r$, consider
the linear operator $\Phi_r$ induced by sifting out all terms with the exponent
vector divisible by $r$ in the power series expansion of a rational function.
More precisely, $\Phi_r$ acts on $F(\bm{t}) \in \ZZ[\bm{t}^{\pm 1}] :=
\ZZ[t_1^{\relphantom{1}}, t_1^{-1}, \dotsc, t_d^{\relphantom{1}}, t_d^{-1}]$ as
follows:
\[
\text{if} \,\, \sum_{(w_1,\dotsc,w_d) \in \ZZ^d} c_{\bm{w}} \, t_1^{w_1} \dotsb
t_d^{w_d} = \sum_{\bm{w}} c_{\bm{w}} \, \bm{t}^{\bm{w}} =
\frac{F(\bm{t})}{\prod\limits_{1 \leq j \leq n} (1-\bm{t}^{\bm{a}_j})} \quad
\text{then} \quad \sum_{\bm{w}} c_{r \bm{w}} \, \bm{t}^{\bm{w}} = \frac{\Phi_r [
  F(\bm{t}) ]}{\prod\limits_{1 \leq j \leq n} (1-\bm{t}^{\bm{a}_j})} \, .
\]
The goal of this article is to understand $\Phi_r [ F(\bm{t}) ]$ for $r \gg 0$.

To state our result, $\alpha \colon \RR^n \to \RR^d$ denotes the linear map
determined by $A$.  The zonotope $Z$ is the image under $\alpha$ of the unit
hypercube $[0,1]^n \subset \RR^n$.  For each $\bm{u} \in \ZZ^d$, set $P(\bm{u})
:= \alpha^{-1}(\bm{u}) \cap [0,1]^n$.  We say that the map $\alpha$ is
\define{degenerate} if there exists $\bm{u}$ in the boundary of $Z$ such that
$\dim P(\bm{u}) = n-d$.  By identifying the rational polytope $P(\bm{u})$ with a
translate $P(\bm{u}) +\bm{x}$ lying in $\ker(A) = \alpha^{-1}(\bm{0})$, the
\define{normalized volume} $\vol_{n-d}\bigl( P(\bm{u}) \bigr)$ equals $(n-d)!$
times the Euclidean volume of $P(\bm{u}) + \bm{x}$ with respect to the lattice
$\alpha^{-1}(\bm{0}) \cap \ZZ^n$.  Let $m$ be the greatest common divisor of the
maximal minors of $A$; in other words, the sublattice $\ZZ A \subseteq \ZZ^d$
generated by the columns of $A$ has index $m$.

\begin{theorem}
  \label{thm:main}
  
  If $\alpha$ is non-degenerate and $F(\bm{t}) = \sum_{\bm{v} \in \ZZ A}
  f_{\bm{v}} \, \bm{t}^{\bm{v}} \in \ZZ[\bm{t}^{\pm 1}]$, then we have
  \begin{xalignat*}{2}
    \limsup_{r \to \infty} \, \frac{\Phi_{r} [ F(\bm{t}) ]}{r^{n-d}} &= \biggl(
    \frac{F(\bm{1})}{(n-d)!} \biggr) \, K_A(\bm{t}) \,\, &\text{where} \quad
    K_A(\bm{t}) &:= \!\! \sum_{\bm{u} \in \interior(Z) \cap \ZZ^d} \!\!
    \vol_{n-d}\bigl( P(\bm{u}) \bigr) \, \bm{t}^{\bm{u}} \, .
  \end{xalignat*}
  Moreover, the coefficients of $K_A(\bm{t})$ are log-concave, quasi-concave,
  and sum to $m^{n-d} \, (n-d)!$.
\end{theorem}

When $A = [ 1 \, \dotsb \,\, 1]$, Theorem~\ref{thm:main} specializes to
Theorem~5.1 in \cite{DF2}.  In this case, $K_A(\bm{t})$ is the Eulerian
polynomial $\sum_{i \geq 0} \eulerian{n-1}{i} \, t^{i+1}$ where
$\eulerian{n-1}{i}$ counts the permutations of $\{1, \dotsc, n-1\}$ with exactly
$i$ ascents.  If $A$ is a totally unimodular matrix ({\footnotesize i.e.{} each
  subdeterminant of $A$ is $ \pm 1$ or $0$}), then we have $K_A(\bm{t}) \in
\ZZ[\bm{t}^{\pm 1}]$; see Remark~\ref{rm:unimodular}.  Thus, for any appropriate
matrix $A$, one might regard $K_A(\bm{t})$ as a generalization of the Eulerian
polynomial.  However, $K_A(\bm{t})$ is not obviously related to the multivariate
Eulerian polynomials in \cite{BHVW}*{\S4.3} or the mixed Eulerian numbers in
\cite{Postnikov}*{\S16}.  Nevertheless, the proof of Theorem~\ref{thm:main}
({\footnotesize see Step~\hyperlink{step1}{1} in \S\ref{sec:two}}) implies that,
for any $F(\bm{t}) = \sum_{\bm{v} \in \interior(mZ)} f_{\bm{v}} \bm{t}^{\bm{v}}
\in \ZZ[\bm{t}^{\pm 1}]$ satisfying $F(\bm{1}) > 0$, there exists $r_0 \in \NN$
such that, for all $r \geq r_0$, $\Phi_{m r} [F(\bm{t}) ]$ has nonnegative
coefficients that are both log-concave and quasi-concave.  Since quasi-concavity
is the multivariate version of unimodality ({\footnotesize see
  Step~\hyperlink{step2}{2} in \S\ref{sec:two}}), Theorem~\ref{thm:main}
generalizes those parts of Theorem~1.4 in \cite{BrentiWelker} and Theorem~1.2 in
\cite{BeckStapledon} that do not make explicit reference to the roots of
$F(\bm{t})$.  It is an open problem to effectively bound $r_0$.

Despite sharing most properties of the Eulerian polynomial, the Laurent
polynomial $K_A(\bm{t})$ need not satisfy the multivariate analogues of
real-zero univariate polynomials.  For example, if
\[
A = \left[ 
  \renewcommand{\arraystretch}{0.7}
  \begin{array}{rrrrr}
    1 & -1 & 1 & 0 & 0 \\
    0 & 1 & -1 & 1 & 0 \\
    0 & 0 & 1 & -1 & 1
  \end{array}
\right]
\] 
then we have $K_A(t_1,t_2,t_3) = t_1t_3+t_2$ which does not have the half-plane
property; see Theorem~3.2 in \cite{Branden}.  Notably, the polynomial
$K_A(t_1,t_2,t_3) = t_1 t_3 + t_2$ is neither real stable nor Hurwitz stable.

In contrast with \cite{BrentiWelker}, \cite{BeckStapledon} or \cite{DF2}, we
provide a new proof that the Eulerian numbers $\eulerian{n}{i}$ are log-concave
and unimodal.  We use polyhedral geometry and the Pr\'ekopa-Leindler inequality
({\footnotesize a.k.a.{} the functional form of the Brunn-Minkowski inequality}) to
establish the log-concavity for the coefficients of $K_A(\bm{t})$; see
Proposition~\ref{p:logconcave}.  In view of \cite{Stanley2} and \cite{Brenti},
our approach to log-concavity appears to be novel within algebraic
combinatorics.

\vspace{1em}

\currentpdfbookmark{Motivation}{sub:motivation} 

\noindent
Our primary motivation for Theorem~\ref{thm:main} comes from commutative algebra
and algebraic geometry.  The matrix $A$ defines a positive $\ZZ^d$-grading on
the polynomial ring $S := \CC[x_1,\dotsc,x_n]$ by setting $\deg(x_j) :=
\bm{a}_j$ for all $1 \leq j \leq n$; see Definition~8.7 in
\cite{MillerSturmfels}.  For any finitely generated $\ZZ^d$-graded $S$-module $M
= \bigoplus_{\bm{w} \in \ZZ^d} M_{\bm{w}}$, the Hilbert series of $M$ can be
expressed uniquely as a rational function of the form $F(\bm{t})/\bigl( \prod_{1
  \leq j \leq n} (1 - \bm{t}^{\bm{a}_j}) \bigr)$; see Theorem~8.20 in
\cite{MillerSturmfels}.  The numerator $F(\bm{t})$ is the \define{Poincar\'e
  polynomial} or \define{$K$\nobreakdash-polynomial} of $M$ --- it records the
alternating sum of the multigraded Betti numbers for $M$; see Definition~8.21
and Theorem~8.23 in \cite{MillerSturmfels}.  Applying the operator $\Phi_r$ to
the Hilbert series of $M$ yields the Hilbert series of the $r$-th Veronese
submodule $\bigoplus_{\bm{w} \in \ZZ^d} M_{r \bm{w}}$.  Theorem~\ref{thm:main}
shows that there exists a unique asymptotic $K$-polynomial for each $M$; see the
first subsection in \S\ref{sec:three} for the important case in which $F(\bm{1})
= 0$.  Therefore, Theorem~\ref{thm:main} gives some information about the
solution to Problem~5.3 in \cite{EinLazarsfeld}.  The results in both
\cite{EinLazarsfeld} and this paper suggest that the asymptotic structure is
surprisingly insensitive to the specific module.

Geometrically, the $\ZZ^d$-grading on $S$ corresponds to the action of the torus
$(\CC^*)^d$ on $\AA^n$ and a multigraded $S$-module $M$ corresponds to a
torus-equivariant sheaf on $\AA^n$.  In this setting, the numerator $F(\bm{t})$
is the class in the equivariant $K$-theory of $\AA^n$ represented by the
corresponding sheaf; see Theorem~8.34 in \cite{MillerSturmfels}.  Hence,
Theorem~\ref{thm:main} implies that the Veronese operator $\Phi_r$ distinguishes
the class $K_A(\bm{t})$ within equivariant $K$-theory.  By taking quotients, we
see that the toric variety $\AA^n /\!\!/ (\CC^*)^d$ is equipped with a
distinguished $K$-theory class $K_A(\bm{t})$.  A geometric explanation for this
distinguished $K$-theory class would be interesting.

Our secondary motivation comes from the theory of lattice point enumeration.
Given lattice polytopes $P_1, \dotsc, P_d$ in $\RR^n$ and nonnegative integers
$w_1, \dotsc, w_d$, the number of lattice points in the Minkowski sum of the
$w_i$-dilates of the $P_i$ is denoted $\bigl| (w_1 P_1 + \dotsb + w_d P_d ) \cap
\ZZ^d \bigr|$.  Ehrhart theory implies that the generating series $\sum_{(w_1,
  \dotsc, w_d)} \bigl| (w_1 P_1 + \dotsb + w_d P_d ) \cap \ZZ^d \bigr| \,
t_1^{w_1} \dotsb \, t_{d}^{w_d}$ can be expressed in the form $F(\bm{t})/\bigl(
\prod_{1 \leq j \leq n} (1 - \bm{t}^{\bm{a}_j}) \bigr)$.  The numerator
$F(\bm{t})$ is a multivariate \define{$h^*$-vector} ({\footnotesize a.k.a.{} Ehrhart
  $h$-vector or $\delta$-polynomial}) for the collection $P_1, \dotsc, P_d$.
Thus, Theorem~\ref{thm:main} also yields a multivariate analogue of
Corollary~1.3 in \cite{BeckStapledon}.

Multivariate formal power series of the form $F(\bm{t})/\bigl( \prod_{1 \leq j
  \leq n} (1 - \bm{t}^{\bm{a}_j}) \bigr)$ also arise naturally in many other
areas of mathematics.  Perhaps the most ubiquitous source is the vector
partition function $\psi_A \colon \NN^d \to \NN$ associated to $A$;
$\psi_A(\bm{u})$ counts the number of nonnegative integer vectors $\bm{x} \in
\NN^n$ such that $A \bm{x} = \bm{u}$.  The generating series $\sum_{\bm{u}}
\psi_A(\bm{u}) \, \bm{t}^{\bm{u}}$ equals $1/\bigl( \prod_{1 \leq j \leq n} (1 -
\bm{t}^{\bm{a}_j}) \bigr)$.  As \cite{Sturmfels} indicates, vector partition
functions appear in representation theory, approximation theory, and statistics.
Reinterpreting $\Phi_r$ in each of these areas will yield new insights into
$K_A(\bm{t})$.

\vspace{1em}

\currentpdfbookmark{Examples}{sub:examples}

\noindent
We end this section with two examples illustrating the necessity of the
hypotheses in Theorem~\ref{thm:main}.

\begin{example}
  \renewcommand{\qedsymbol}{$\diamond$}

  Let $A = \left[
    \begin{smallmatrix}
      2 & 1 & 0 \\
      0 & 1 & 2 
    \end{smallmatrix} 
  \right]$, so that $d = 2$, $n = 3$ and $m = \gcd(2,4,2) = 2$.  For all $r >
  2$, we have $\Phi_{2r}[1] = (r-1) t_1^2 t_2^2 + (r-1) t_1^2 t_2 + (r-1) t_1
  t_2^2 + r t_1 t_2 + t_1 + t_2 +1$ which implies that 
  \[
  \lim\limits_{r \to \infty} \frac{\Phi_{2r}[1]}{2r} = \tfrac{1}{2} t_1^2 t_2^2
  + \tfrac{1}{2} t_1^2 t_2 + \tfrac{1}{2} t_1 t_2^2 + \tfrac{1}{2} t_1 t_2 \, .
  \]  
  However, we also have $\Phi_{2r + 1}[1] = r t_1^2 t_2^2 + r t_1 t_2 + 1$ which
  shows that $\lim\limits_{r \to \infty} \Phi_{r}[1] /r$ does not exist.

  The zonotope $Z$ associated to $A$ is $\conv \bigl\{ (0,0), (2,0), (0,2),
  (3,1), (1,3), (3,3) \bigr\}$.  Its interior lattice points are $(1,1)$,
  $(2,1)$, $(1,2)$, $(2,2)$ and the associated polytopes are the line segments:
  \begin{xalignat*}{2}
    P(1,1) &= \conv \bigl\{ (0,1,0), (\tfrac{1}{2},0, \tfrac{1}{2}) \bigr\} \, ,
    & P(2,1) &= \conv \bigl\{ (\tfrac{1}{2},1,0), (1,0,\tfrac{1}{2}) \bigr\} \,
    , \\
    P(1,2) &= \conv \bigl\{ (0,1,\tfrac{1}{2}), (\tfrac{1}{2},0, 1) \bigr\} \, ,
    & P(2,2) &= \conv \bigl\{ (1,0,1), (\tfrac{1}{2},0, 1) \bigr\} \, .
  \end{xalignat*}
  Since $\alpha^{-1}(\bm{0}) \cap \ZZ^3 = \{ (i,-2i,i) : i \in \ZZ \}$ and the
  coefficient of $\bm{t}^{\bm{u}}$ in $K_A(\bm{t})$ is the normalized volume of
  $P(\bm{u})$, we have $K_A(t_1,t_2) = \tfrac{1}{2} t_1^2 t_2^2 + \tfrac{1}{2}
  t_1^2 t_2 + \tfrac{1}{2} t_1 t_2^2 + \tfrac{1}{2} t_1 t_2$.  The coefficients
  sum to $2$.  \qed
\end{example}

\begin{example}
  \renewcommand{\qedsymbol}{$\diamond$}

  Let $A = \left[
    \begin{smallmatrix}
      1 & 1 & 0 \\
      0 & 0 & 1 
    \end{smallmatrix} 
  \right]$, so that $d = 2$, $n = 3$ and $m = 1$.  It follows that
  \begin{xalignat*}{2}
    \limsup_{r \to \infty} \frac{\Phi_{r}[1]}{r} &= \lim_{r \to \infty}
    \frac{(r-1)t_1 + 1}{r} = t_1 \, , & \limsup_{r \to \infty}
    \frac{\Phi_{r}[t_1]}{r} &= \lim_{r \to \infty}
    \frac{(r-1)t_1}{r} = t_1 \, , \\
    \limsup_{r \to \infty} \frac{\Phi_{r}[t_2]}{r} &= \lim_{r \to \infty}
    \frac{(r-1)t_1 t_2 + t_2}{r} = t_1 t_2 \, , & \limsup_{r \to \infty}
    \frac{\Phi_{r}[t_1 t_2]}{r} &= \lim_{r \to \infty} \frac{r t_1 t_2}{r} = t_1
    t_2 \, .
  \end{xalignat*}
  Although each of the limits exists, they do not simply depend up to a scalar
  on the matrix $A$.  The zonotope $Z = \conv \bigl\{ (0,0), (2,0), (0,1), (2,1)
  \bigr\}$ has no interior lattice points, but the polytopes $P(1,0) = \conv
  \bigl\{ (1,0,0), (0,1,0) \bigr\}$ and $P(1,1) = \conv \bigl\{ (1,0,1), (0,1,1)
  \bigr\}$ have dimension $n-d = 1$.  In particular, the map $\alpha$ is
  degenerate.  \qed
\end{example}

Further examples, open problems, and other connections are discussed in
\S\ref{sec:three}.  The proof of Theorem~\ref{thm:main} is given in
\S\ref{sec:two}.

\subsubsection*{Acknowledgements}

We thank Matthias Beck, Daniel Erman, and Alan Stapledon for stimulating our
interest in asymptotic Hilbert series and providing feedback on a preliminary
version of this document.  We also thank Mike Roth for his valuable comments and
insights.  The computer software \emph{Macaulay2}~\cite{M2} was useful for
generating examples.  Both authors were partially supported by NSERC.

\section{The Proof of the Main Theorem}
\label{sec:two}

\noindent
We divide the proof of Theorem~\ref{thm:main} into three steps:
Step~\hyperlink{step1}{1} establishes that the limit exists and provides a
polyhedral interpretation for $K_A(\bm{t})$, Step~\hyperlink{step2}{2} uses the
polyhedral interpretation to prove the log-concavity of the coefficients, and
Step~\hyperlink{step3}{3} gives a geometric explanation for the sum of the
coefficients.  For brevity, we write $\prod_j := \prod_{1 \leq j \leq n}$ and
$\sum_j := \sum_{1 \leq j \leq n}$.

\subsection*{Step~1}
\label{step1}
\hypertarget{step1}{}

\renewcommand{\qedsymbol}{$\rhd$}

To begin, we describe the matrix associated to the linear operator $\Phi_r$.
Despite being defined via multivariate formal power series, $\Phi_r$ may be
understood in terms of a linear operator on Laurent polynomials.  Specifically,
let $\Psi_r \in \operatorname{End} \bigl( \ZZ[\bm{t}^{\pm 1}] \bigr)$ be the
linear operator that discards the terms with exponent vectors that are not
componentwise divisible by $r$ and divides each of the remaining exponent
vectors by $r$.  Consider a Laurent polynomial $F(\bm{t}) \in \ZZ[\bm{t}^{\pm
  1}]$ such that $\sum_{\bm{w}} c_{\bm{w}} \, \bm{t}^{\bm{w}} = F(\bm{t}) /
\bigl( \prod_j (1-\bm{t}^{\bm{a}_j}) \bigr)$.  The linear operator $\Psi_r$
lifts to an endomorphism on multivariate formal power series.  Applying $\Psi_r$
to this rational function yields
\begin{align*}
  \frac{\Phi_r [ F(\bm{t}) ]}{\prod_j (1 - \bm{t}^{\bm{a_j}})} =
  \sum_{\bm{w}} c_{r \bm{w}} \, \bm{t}^{\bm{w}} &= \Psi_r \left[
    \frac{F(\bm{t})}{\prod_{j} (1-\bm{t}^{\bm{a}_j})} \right] 
   = \Psi_r \left(
    \frac{F(\bm{t}) \, \prod_{j} (1 + \bm{t}^{\bm{a}_j} + \bm{t}^{2 \bm{a}_j} +
      \dotsb + \bm{t}^{(r-1) \bm{a}_j}) }{\prod_{j} (1-\bm{t}^{r \bm{a}_j})}
  \right) \\
  &= \frac{\Psi_r [ F(\bm{t}) \, \prod_{j} (1 + \bm{t}^{\bm{a}_j} +
    \bm{t}^{2 \bm{a}_j} + \dotsb + \bm{t}^{(r-1) \bm{a}_j}) ]}{\prod_{j}
    (1-\bm{t}^{\bm{a}_j})} \, .
\end{align*}
Hence, we have $\Phi_r [ F(\bm{t}) ] = \Psi_r [ F(\bm{t}) \, \prod\nolimits_{j}
(1 + \bm{t}^{\bm{a}_j} + \bm{t}^{2 \bm{a}_j} + \dotsb + \bm{t}^{(r-1) \bm{a}_j})
]$ which is a multivariate version of Lemma~3.2 in \cite{BeckStapledon}.  To
express $\Phi_r$ as a matrix with respect to the monomial basis, set
\[
C_r(\bm{u},\bm{v}) := \bigl| \{ \bm{x} = (x_1, \dotsc, x_n) \in \ZZ^n \cap
[0,r-1]^n : \textstyle\sum\nolimits_{j} x_j \, \bm{a}_j = r \, \bm{u} - \bm{v}
\} \bigr|.
\]  
Since $F(\bm{t}) := \sum_{\bm{v} \in \ZZ A} f_{\bm{v}} \, \bm{t}^{\bm{v}}$, we
have $\Phi_{r} [ F(\bm{t}) ] = \textstyle\sum_{\bm{u}} \bigl(
\textstyle\sum_{\bm{v} \in \ZZ A} C_{r}(\bm{u},\bm{v}) \, f_{\bm{v}} \bigr)
\bm{t}^{\bm{u}}$.  As $v \in \ZZ A$, we also have $C_r(\bm{u},\bm{v}) = 0$ for
all $\bm{u} \not\in \ZZ A$.  However, the sublattice $\ZZ A \subseteq \ZZ^d$ has
index $m$ so $m \bm{u} \in \ZZ A$.  Hence, it is enough to consider $\Phi_{m r}
[ F(\bm{t}) ] = \textstyle\sum_{\bm{u}} \bigl( \textstyle\sum_{\bm{v} \in \ZZ A}
C_{m r}(\bm{u},\bm{v}) \, f_{\bm{v}} \bigr) \bm{t}^{\bm{u}}$.

We next relate the integer coefficients $C_{m r}(\bm{u},\bm{v})$ to the
normalized volume of the rational polytope $P(\bm{u}) := \{ \bm{x} \in [0,1]^n :
\textstyle\sum_j x_j \, \bm{a}_j = \bm{u} \} = \alpha^{-1}( \bm{u}) \cap
[0,1]^{n} \subseteq \RR^n$.  For any $\bm{u} \in \ZZ^d$ and any $\bm{v} \in \ZZ
A$, there exists $\bm{y}$ and $\bm{z} \in \ZZ^n$ such that $m \bm{u} =
\alpha(\bm{y})$ and $\bm{v} = \alpha(\bm{z})$.  It follows that 
\[
C_{m r}(\bm{u},\bm{v}) = \bigl| \bigl( (mr-1) P(\bm{u}) + \tfrac{1}{m} \bm{y} -
\bm{z} \bigr) \cap \ZZ^n \bigr| = \bigl| \bigl( (mr-1) P(\bm{u}) + \tfrac{1}{m}
\bm{y} \bigr) \cap \ZZ^n \bigr| \, .
\]
The enumerative theory for systems of linear diophantine equations establishes
that the function
$r \mapsto \bigl| \bigl( (mr-1) P(\bm{u}) +
\tfrac{1}{m} \bm{y} \bigr) \cap \ZZ^n \bigr| = \bigl| \bigl( (r-1) \bigl( m
P(\bm{u}) \big) + (m-1)P(\bm{u}) + \tfrac{1}{m} \bm{y} \bigr) \cap \ZZ^n \bigr|$
is a quasi-polynomial; see Theorem~6.50 in \cite{BrunsGubeladze} or
Theorem~4.6.11 in \cite{Stanley}.  Moreover, the leading coefficient of this
quasi-polynomial is the relative volume of $m P(\bm{u})$ with respect to the
lattice $\alpha^{-1}(rmu) \cap \ZZ^n$; see Theorem~6.55 in
\cite{BrunsGubeladze}.  Since $m \bm{u} \in \ZZ A$ and $(mr-1) P(\bm{u}) +
\tfrac{1}{m} \bm{y} \subset \alpha^{-1}(rmu)$, the normalized volume
$\vol_{n-d}(P(\bm{u}))$ equals $ (n-d)! / m^{n-d}$ times coefficient of the
degree $n-d$ term in this quasi-polynomial, so
\begin{align*}
  \lim_{r \to \infty} \frac{C_{m r} (\bm{u},\bm{v})}{(m r)^{n-d}} &= \lim_{r \to
    \infty} \frac{\bigl| \bigl( (mr-1) P(\bm{u}) + \tfrac{1}{m} \bm{y} \bigr)
    \cap \ZZ^n \bigr|}{(mr)^{n-d}} = \frac{\vol_{n-d} \bigl( P(\bm{u})
    \bigr)}{(n-d)!} \, .
\end{align*}
By hypothesis, $\alpha$ is non-degenerate, so $P(\bm{u})$ has dimension $n-d$ if
and only if $\bm{u}$ lies in the interior of $Z$.  Therefore, we have
\begin{align*}
  \lim_{r \to \infty} \frac{\Phi_{m r} [ F(\bm{t}) ]}{(mr)^{n-d}} &=
  \sum_{\bm{u}} \Biggl( \sum_{\bm{v} \in \ZZ A} \biggl( \lim\limits_{r \to
    \infty} \frac{C_{m r} (\bm{u}, \bm{v})}{(m r)^{n-d}} \biggr) f_{\bm{v}}
  \Biggr) \bm{t}^{\bm{u}} = \frac{F(\bm{1})}{(n-d)!} \Biggl( \sum_{\bm{u} \in
    \interior(Z) \cap \ZZ^d} \vol_{n-d} \bigl( P(\bm{u}) \bigr) \,
  \bm{t}^{\bm{u}} \Biggr) \, .
\end{align*}
It follows immediately from this polyhedral description that $K_A(\bm{t})$
inherits the symmetries of the zonotope $Z$. \qed

\subsection*{Step~2}
\hypertarget{step2}{}

The second step takes advantage a well-known theorem from convex geometric
analysis.  To isolate the applicable result, recall that $g \colon \RR^d \to
\RR$ is \define{log-concave} if $g(\bm{u}) > 0$ and $\log(g)$ is concave.
Equivalently, for $s \in [0,1]$ and $\bm{u}, \bm{v} \in \ZZ^d$, we have the
inequality $g \bigl( s \bm{u} + (1-s) \bm{v} \bigr) \geq g(\bm{u})^s
g(\bm{v})^{1-s}$.

\begin{proposition}
  \label{p:logconcave}
  Let $X \subset \RR^n$ be a closed convex set and let $\pi \colon \RR^n \to
  \RR^d$ be a surjective linear map.  If the function $g \colon \RR^d \to \RR$
  assigns to $\bm{u} \in \RR^{d}$ the volume of the fibre $\pi^{-1}(\bm{u}) \cap
  X$, then $g$ is log-concave.
\end{proposition}

\begin{proof}
  \renewcommand{\qedsymbol}{$\Box$}

  Choosing a basis for $\ker(\pi)$ defines an injective linear map $\varphi
  \colon \RR^{n-d} \to \RR^n$ such that $\pi \circ \varphi$ is the zero map.  We
  may also choose an injective linear map $\eta \colon \RR^d \to \RR^n$ such
  that $\pi \circ \eta$ is the identity map, because $\pi$ is surjective.  For
  any $\bm{u} \in \RR^d$, we set
  \[
  X(\bm{u}) := \{ \bm{p} \in \RR^{n-d} : \varphi(\bm{p}) + \eta(\bm{u}) \in X \}
  \iso \pi^{-1}(\bm{u}) \cap X \, ,
  \]  
  so that $g(\bm{u}) = \vol\bigl( X(\bm{u}) \bigr)$.  Since $X$ is convex, we
  have $s \, X(\bm{u}) + (1-s) \, X(\bm{v}) \subseteq X\bigl( s \bm{u} + (1-s)
  \bm{v} \bigr)$ for all $s \in [0,1]$ and $\bm{u}, \bm{v} \in \RR^d$.  If $1_Y
  \colon \RR^{n-d} \to \RR$ denotes the indicator function for $Y \subseteq
  \RR^{n-d}$, i.e.{} $1_Y(\bm{y}) = \left\{ 
    \begin{smallmatrix}
      1 & \text{if $\bm{y} \in Y$ \,\!} \\
      0 & \text{if $\bm{y} \not\in Y$,}
    \end{smallmatrix} 
  \right.$ then we obtain $1_{X( s \bm{u} + (1-s) \bm{v} )} \bigl( s \bm{p} +
  (1-s) \bm{q} \bigr) \geq \bigl( 1_{X(\bm{u})}(\bm{p}) \bigr)^s \bigl(
  1_{X(\bm{v})}(\bm{q}) \bigr)^{1-s}$ for all $\bm{p}, \bm{q} \in \RR^{n-d}$.
  There exists a positive $k \in \RR$ such that $g(\bm{u}) = k \int_{\RR^{n-d}}
  1_{X(\bm{u})}(\bm{p}) d\bm{p}$, so the Pr\'ekopa-Leindler inequality
  (Theorem~7.1 in \cite{Gardner} or Theorem~6.4 in \cite{Villani}) yields
  \begin{align*}
    g \bigl( s \bm{u} + (1-s) \bm{v} \bigr) &= k \int_{\RR^{n-d}} 1_{X(s \bm{u}
      + (1-s) \bm{v})}(\bm{p}) \; d \bm{p} \\ &\geq k \left( \int_{\RR^{n-d}}
      1_{X(\bm{u})} (\bm{p}) \; d \bm{p} \right)^{s} \left( \int_{\RR^{n-d}}
      1_{X(\bm{v})} (\bm{p}) \; d \bm{p} \right)^{1-s} = \bigl( g(\bm{u})
    \bigr)^{s} \bigl( g(\bm{v}) \bigr)^{1-s} \, . \qedhere
  \end{align*}
\end{proof}

To establish that the coefficient function of $K_A(\bm{t})$ is log-concave, we
simply apply Proposition~\ref{p:logconcave} when $X = [0,1]^n$ and $\pi =
\alpha$, because Step~\hyperlink{step1}{1} shows that the coefficients of
$K_A(\bm{t})$ equal the normalized volume of the fibres of the map $\alpha
\colon [0,1]^n \to Z$.

We finish this step by showing that the coefficient function of $K_A(\bm{t})$ is
quasi-concave.  A function $g \colon \RR^d \to \RR$ is \define{quasi-concave} if,
for all $e \in \RR$, its superlevel set $\{ \bm{u} \in \RR^d : g(\bm{u}) \geq e
\}$ is convex.  Equivalently, for $s \in [0,1]$ and $\bm{u}, \bm{v} \in
\RR^{d}$, we have $g \bigl( s \bm{u} + (1-s) \bm{v} \bigr) \geq \min \bigl(
g(\bm{u}), g(\bm{v}) \bigr)$.  In particular, every positive log-concave
function is quasi-concave. \qed

\subsection*{Step~3}
\hypertarget{step3}{}
\renewcommand{\qedsymbol}{$\Box$}

Lastly, we reinterpret the coefficients of $K_A(\bm{t})$ as the normalized
volumes of certain regions in the hypercube $[0,m]^{n-d}$.  This perspective is
inspired by the well-known interpretation for the Eulerian numbers as the
normalized volume of ``slabs'' in the hypercube; see \cite{Stanley3}.  In fact,
when $A = [ 1 \, \dotsb \,\, 1 ]$, we recover this interpretation.

Without loss of generality, we may assume that $A$ has a specified block
structure.  By reordering the columns and making a unimodular change of basis on
$\ZZ^d$, it suffices to consider the case in which $A = [ \, H \, | \, B' \,]$
where $H$ is a $(d \times d)$-matrix in Hermite normal form, $\det(H) = m$, and
$B' = [ \bm{b}_1 \, \dotsb \,\, \bm{b}_{d} ]^{\textsf{t}}$ is a $\big( d \times
(n-d) \bigr)$-matrix; see \cite{Schrijver}*{\S4.1}.  Let $B$ be the integer
block $(n \times d)$-matrix $\bigl[ \renewcommand{\arraystretch}{0.5}
\renewcommand{\arraycolsep}{0pt}
\begin{array}{c}  
  \scriptstyle{m H^{-1} B'} \\ \scriptstyle{-m I_{n-d}}
\end{array} 
\bigr]$, so that $AB = 0$.  If $m = 1$ then $H = I_d$ and $B$ is the Gale dual
of $A$; see \cite{Ziegler}*{\S6.4}.

For $\bm{u} \in \ZZ^d$, consider regions
\[
R(\bm{u}) := \{ \bm{p} \in [0,m]^{n-d} : \text{$m (u_i-1) \leq \bm{b}_i \cdot
  \bm{p} \leq m u_i$ for all $1 \leq i \leq d$} \} \subset \RR^{n-d} \, .
\]
By definition, each region $R(\bm{u})$ is a rational polytope and, combined
together, these regions partition the hypercube $[0, m]^{n-d}$.  Hence, the
union $\bigcup_{\bm{u}} R(\bm{u})$ has normalized volume equal to $m^{n-d} (n-d)!$.

To complete the proof, it is enough to show that the coefficients of
$K_A(\bm{t})$ also correspond to the normalized volumes of the regions
$R(\bm{u})$.  If $J$ is the integer $(n \times d)$-matrix $\bigl[
\renewcommand{\arraystretch}{0.5} \renewcommand{\arraycolsep}{0pt}
\begin{array}{c}  
  \scriptstyle{m H^{-1}} \\ \scriptstyle{0}
\end{array} 
\bigr]$, then we obtain $AJ = mI_d$.  Moreover, the inequalities $m(u_i-1) \leq
\bm{b}_i \cdot \bm{p} \leq mu_i$ hold if and only if the inequalities $1 \geq
u_i - \tfrac{1}{m} \bm{b}_i \cdot \bm{p} \geq 0$ hold.  It follows that the
affine map $\bm{p} \mapsto \tfrac{1}{m} (J \bm{u} - B \bm{p})$ sends the region
$R(\bm{u})$ into the rational polytope $P(\bm{u}) = \{ \bm{x} \in [0,1]^n :
\sum_j x_j \, \bm{a}_j = \bm{u} \} = \alpha^{-1}(\bm{u}) \cap [0,1]^n$.  The
inverse map is associated to the $\bigl( (n-d) \times n)$-matrix $L := [ \, 0 \,
| \, I_{n-d} \, ]$ which satisfies $LJ = 0$ and $LB = - m I_{n-d}$.  Since
Step~\hyperlink{step1}{1} establishes that the coefficients of $K_A(\bm{t})$
equal the normalized volumes of the polytopes $P(\bm{u})$, we conclude that
$K_A(\bm{1}) = m^{n-d} (n-d)!$ as required. \qed

\begin{remark}
  \label{rm:unimodular}

  Suppose that $A$ is a totally unimodular matrix which means $m = 1$.  If
  $\bm{u} \in \ZZ^d$, then $P(\bm{u})$ is a lattice polytope (see Theorem~19.1
  in \cite{Schrijver}) and $C_r(\bm{u},\bm{v}) = \bigl| \bigl( (r-1)P(\bm{u})
  \bigr) \cap \ZZ^n \bigr|$.  Thus, the normalized volume $\vol_{n-d}\bigl(
  P(\bm{u}) \bigr)$ is a nonnegative integer which implies that $K_A(\bm{t}) \in
  \ZZ[\bm{t}^{\pm 1}]$ and $K_A(\bm{1}) = (n-d)!$.  Finding a combinatorial
  description for these coefficients remains an open problem.  One tantalizing
  possibility is that the coefficients naturally enumerate some partition of the
  symmetric group on $n-d$ letters.
\end{remark}

\section{Other Connections}
\label{sec:three}

\noindent
This section uses two additional examples to highlight some potential
applications for Theorem~\ref{thm:main}.  To help orient future research, we
also state a few open problems.

\subsection*{Asymptotic Equalities}

When $F(\bm{1}) \neq 0$, Theorem~\ref{thm:main} shows that some of the
coefficients of $\Phi_r [ F(\bm{t}) ]$, regarded as functions of $r$, are
asymptotically equal to an explicit multiple of $r^{n-d}$.  In many situations,
Theorem~\ref{thm:main} yields an asymptotic equality even though $F(\bm{1}) =
0$.  Roughly speaking, the Taylor expansion of $F(\bm{t})$ about $\bm{t} =
\bm{1}$ allows one to cancel factors from the denominator and, thereby, apply
Theorem~\ref{thm:main} to a submatrix of $A$.  We demonstrate this approach in
the following prototypical example.

\begin{example}
  \renewcommand{\qedsymbol}{$\diamond$}

  Let $S = \CC[x_1, \dotsc, x_6]$ be the polynomial ring with the positive
  $\ZZ^2$-grading induced by $A = \left[
    \begin{smallmatrix}
      2 & 1 & 2 & 0 & 1 & 2 \\
      0 & 1 & 1 & 2 & 2 & 2
    \end{smallmatrix}
  \right]$.  Observe that $m = 1$.  Fix the finitely generated $\ZZ^2$-graded
  $S$-module $M = S/I$ where $I = \langle x_5^2 - x_4 x_6, x_3 x_5 - x_2 x_6,
  x_3 x_4 - x_2 x_5, x_3^2 - x_1 x_6, x_2 x_3 - x_1 x_5, x_2^2 - x_1 x_6
  \rangle$.  Under the standard grading ({\footnotesize i.e.{} when $A = [ \, 1 \,
    \dotsb \, 1 \, ]$}), the module $M$ would be the homogeneous coordinate ring
  of the Veronese surface in $\PP^5$.  As a rational function, the Hilbert
  series of $M$ is
  \begin{align*}
    \frac{F(\bm{t})}{\prod_{j}(1 - \bm{t}^{\bm{a}_{j}})} &=
    \frac{(1-t_1 t_2)(1-t_1^2 t_2)(1-t_1 t_2^2)(1 + t_1 t_2 + t_1^2 t_2 + t_1
      t^2)}{(1-t_1^2)(1-t_1 t_2)(1-t_1^2 t_2)(1-t_2^2)(1-t_1 t_2^2)(1-t_1^2
      t_2^2)} \, .
  \end{align*}
  Since $\ell := \operatorname{codim}(M) = 3 > 0$, it follows that $F(\bm{1}) =
  0$.

  In this context, an appropriate expansion of $F(\bm{t})$ comes from a choice
  of initial module.  Each monomial initial module of $M$ is homogeneous with
  respect to the $\ZZ^n$-grading arising from the identity matrix $I_n$; see
  \cite{MillerSturmfels}*{\S2.2}.  Since $M$ and any initial module have the
  same Hilbert series, the Taylor expansion for the numerator of the
  $\ZZ^n$-graded Hilbert series produces an expression for $F(\bm{t})$ as a
  polynomial in the variables $(1-\bm{t}^{\bm{a}_1}), \dotsc,
  (1-\bm{t}^{\bm{a}_n})$; see Exercise~8.15 in \cite{MillerSturmfels}.
  Moreover, the lowest terms in each such expression have degree $\ell$, are
  square-free, and have nonnegative coefficients; see Exercise~8.8 in
  \cite{MillerSturmfels}.  For instance, the monomial initial ideals $\langle
  x_2, x_3, x_5 \rangle^2$ and $\langle x_1, x_2, x_5^2 \rangle \cap \langle
  x_2^2, x_5, x_6 \rangle$ of $I$ ({\footnotesize chosen from among the $29$
    possibilities}) yield the following expansions:
  \begin{align*}
    F(\bm{t}) &= 4 (1 \! - \! \bm{t}^{\bm{a}_2}) (1 \! - \! \bm{t}^{\bm{a}_3})
    (1 \! - \!  \bm{t}^{\bm{a}_5}) - (1 \! - \!  \bm{t}^{\bm{a}_2})^2 (1 \! - \!
    \bm{t}^{\bm{a}_3}) (1 \! - \!  \bm{t}^{\bm{a}_5}) - (1 \! - \!
    \bm{t}^{\bm{a}_2}) (1 \! - \!  \bm{t}^{\bm{a}_3})^2 (1 \! - \!
    \bm{t}^{\bm{a}_5}) \\&\relphantom{==} - (1 \! - \!  \bm{t}^{\bm{a}_2}) (1 \!
    - \!
    \bm{t}^{\bm{a}_3}) (1 \! - \! \bm{t}^{\bm{a}_5})^2 \\
    &= 2 (1 \! - \! \bm{t}^{\bm{a}_1}) (1 \! - \! \bm{t}^{\bm{a}_2}) (1 \! - \!
    \bm{t}^{\bm{a}_5}) + 2 (1 \! - \! \bm{t}^{\bm{a}_2}) (1 \! - \!
    \bm{t}^{\bm{a}_5}) (1 \! - \! \bm{t}^{\bm{a}_6}) - (1 \!  - \!
    \bm{t}^{\bm{a}_1}) (1 \! - \! \bm{t}^{\bm{a}_2}) (1 \! - \!
    \bm{t}^{\bm{a}_5})^2 \\&\relphantom{==} - (1 \! - \!  \bm{t}^{\bm{a}_1}) (1
    \! - \!  \bm{t}^{\bm{a}_2}) (1 \! - \!  \bm{t}^{\bm{a}_5}) (1 \! - \!
    \bm{t}^{\bm{a}_6}) - (1 \! - \!  \bm{t}^{\bm{a}_2})^2 (1 \! - \!
    \bm{t}^{\bm{a}_5}) (1 \! - \!  \bm{t}^{\bm{a}_6}) \, .
  \end{align*}
  Observe that the coefficient of the term $(1- \bm{t}^{\bm{a}_{s_1}}) (1-
  \bm{t}^{\bm{a}_{s_2}}) (1- \bm{t}^{\bm{a}_{s_3}})$ equals the multiplicity
  $\mu_{\bm{s}}$ of the minimal prime $\langle x_{s_1}, x_{s_2}, x_{s_3}
  \rangle$ for the initial ideal; cf. Definition~8.43 in \cite{MillerSturmfels}.
  Since we have $\Phi_r [ (1-\bm{t}^{\bm{a}_i}) F(\bm{t}) ] =
  (1-\bm{t}^{\bm{a}_i}) \Phi_r [ F(\bm{t}) ]$, Theorem~\ref{thm:main} gives
  \[
  \lim_{r \to \infty} \, \frac{\Phi_r [ F(\bm{t}) ]}{r^{n-\ell-d}} =
  \frac{1}{(n-\ell-d)!} \sum_{\bm{s}} \left(
    \frac{\mu_{\bm{s}}}{m_{A_{\widehat{\bm{s}}}}^{n-\ell-d}} \,
    K_{A_{\widehat{\bm{s}}}}(\bm{t}) \prod_{1 \leq i \leq \ell}
    (1-\bm{t}^{\bm{a}_{s_i}}) \right) \, ,
  \]
  where $A_{\widehat{\bm{s}}}$ is the $(d \times (n-\ell))$-submatrix of $A$ in
  which the columns indexed by $\bm{s}$ are omitted and
  $m_{A_{\widehat{\bm{s}}}}$ is the greatest common divisor of the maximal
  minors of $A_{\widehat{\bm{s}}}$.  In particular, we have
  \begin{align*}
    \lim_{r \to \infty} & \frac{\Phi_r [ F(\bm{t}) ]}{r} = \tfrac{1}{2} t_1 t_2
    (t_1+1) (t_2+1) (t_1 t_2+1) (1 - t_1 t_2) (1 - t_1^2 t_2) (1 -  t_1 t_2^2)
    \\
    &= \tfrac{4}{4} \Bigl( K_{\left[
          \begin{smallmatrix}
            2 & 0 & 2 \\
            0 & 2 & 2 
          \end{smallmatrix}
        \right]} (\bm{t}) \Bigr) (1 \! - \! \bm{t}^{\bm{a}_2}) (1 \! - \!
      \bm{t}^{\bm{a}_3}) (1 \! - \!
      \bm{t}^{\bm{a}_5}) \\
      &= \tfrac{2}{2} \Bigl( K_{\left[
          \begin{smallmatrix}
            2 & 0 & 2 \\
            1 & 2 & 2 
          \end{smallmatrix}
        \right]}(\bm{t}) \Bigr) (1 \! - \! \bm{t}^{\bm{a}_1}) (1 \! - \!
      \bm{t}^{\bm{a}_2}) (1 \! - \! \bm{t}^{\bm{a}_5}) 
      + \tfrac{2}{2} \Bigl( K_{\left[
          \begin{smallmatrix}
            2 & 2 & 0 \\
            0 & 1 & 2 
          \end{smallmatrix}
        \right]}(\bm{t}) \Bigr) (1 \! - \! \bm{t}^{\bm{a}_2}) (1 \! - \!
      \bm{t}^{\bm{a}_5}) (1 \! - \! \bm{t}^{\bm{a}_6}) \, .
  \end{align*}
  As the expansions for $F(\bm{t})$ vary, these equations lead to nontrivial
  relations among asymptotic $K$-polynomials associated to submatrices of $A$.
  \qed
\end{example}

This approach also suggests a way to understand lower-order asymptotics.
Specifically, if $\limsup\limits_{r \to \infty} \frac{\Phi_r [ F (\bm{t})
  ]}{r^{n-\ell-d}} = \frac{G(\bm{t})}{(n-\ell-d)!}$, then the Laurent polynomial
$(n-\ell-d)! \, \Phi_r [ F (\bm{t}) ] - G(\bm{t})$ vanishes at $\bm{t} = \bm{1}$
and one should analyze $\limsup\limits_{r \to \infty} \frac{(n-\ell-d)! \,
  \Phi_r [ F (\bm{t}) ] - G(\bm{t})}{r^{n-\ell-d-1}}$.  Providing an algebraic
or geometric interpretation for this lower-order asymptotics is an open problem.

\subsection*{Stochastic Matrices}

By scaling the matrix associated to the linear operator $\Phi_r$, we obtain a
stochastic matrix.  Specifically, the matrix $C(r) := r^{\, d-n} \, [ \,
C_{r}(\bm{u},\bm{v}) \, ]$ where $\bm{u}$ and $\bm{v}$ range over $\interior(Z)
\cap \ZZ^d$ is a nonnegative square matrix each of whose columns sum to $1$.
When $A = [ \, 1 \, \dotsb \, 1 \, ]$, \cite{Holte} ({\footnotesize also see
  \cite{DF1}}) establishes that $C(r)$ has the following ``amazing'' properties:
\begin{itemize}
\item for all $r$, the stationary vector ({\footnotesize i.e.{} the eigenvector
    with eigenvalue $1$}) corresponds to the coefficients of $K_A(\bm{t)} /
  (n-d)!$;
\item the matrix $C(r)$ has eigenvalues $r^{-i}$ for $0 \leq i < n-d$ with
  explicit eigenvectors independent of $r$; and
\item we have $C(r_1) C(r_2) = C(r_1 r_2)$. 
\end{itemize}
The next example illustrates how these properties extend to our more general
setting.

\begin{example}
  \renewcommand{\qedsymbol}{$\diamond$}

  Let $A = \left[
    \begin{smallmatrix}
      1 & 1 & 0 & 0 & -1 \\
      0 & 0 & 1 & 1 & 1
    \end{smallmatrix} 
  \right]$, so that $Z = \conv \{ (0,0), (2,0), (-1,1), (2,2), (-1,3), (1,3) \}$
  and $K_A(t_1,t_2) = t_1 t_2^2 + 2 t_1 t_2 + 2 t_2^2 + t_2$.  Fixing $(1,2),
  (1,1), (2,0), (1,0)$ as the ordering for the interior lattices points yields
  the stochastic matrix
  \[
  C(r) := r^{-3} \left[ 
    \begin{matrix}
      \binom{r+2}{3} & \binom{r+1}{3} & \binom{r+1}{3} & \binom{r}{3} \\
      2 \binom{r+1}{3} & 2\binom{r+1}{3} + \binom{r+1}{2} & 2\binom{r}{3} +
      \binom{r}{2} & 2 \binom{r+1}{3} \\
      2 \binom{r+1}{3} & 2\binom{r}{3} + \binom{r}{2} & 2\binom{r+1}{3} +
      \binom{r+1}{2}  & 2 \binom{r+1}{3} \\
      \binom{r}{3} & \binom{r+1}{3} & \binom{r+1}{3} & \binom{r+2}{3}
    \end{matrix}
    \right] \, .
  \]
  The eigenvalues of $C(r)$ are $1, r^{-1}, r^{-1}, r^{-2}$ and the
  corresponding eigenvectors are simply $[ \, 1 \,\, 2 \,\, 2 \,\, 1 \,
  ]^{\textsf{t}}$, $[ \, 1 \,\, 0 \,\, 0 \,\, -1 \, ]^{\textsf{t}}$, $[ \, 0
  \,\, 1 \,\, -1 \,\, 0 \, ]^{\textsf{t}}$, $[ \, 1 \,\, -1 \,\, -1 \,\, 1 \,
  ]^{\textsf{t}}$. \qed
\end{example}

In the standard graded case, \cite{Holte} and \cite{DF1} relate the matrix
$C(r)$ to the process of ``carries'' when adding integers and to shuffling
cards, respectively.  Do our more general stochastic matrices also correspond to
known Markov chains?  Regardless, it would be interesting to bound the rates of
convergence for the associated Markov chains and thereby extend the results in
\cite{DF2}*{\S3}.

\raggedright

\end{document}